\documentclass[12pt]{article}
\usepackage{authblk} % Adds affiliation

\usepackage[T1]{fontenc}
\usepackage[utf8]{inputenc}
\usepackage[english]{babel}
\usepackage{enumerate,amsmath,amsthm}
\usepackage{amssymb}
\usepackage{bbding}         % various symbols (squares, asterisks, scissors, ...)
\usepackage{bm}             % boldface symbols (\bm)
\usepackage{bbm} % \mathbb{1}
\usepackage{bookman}
\usepackage{comment}%for commentiong out larger portions of text
\usepackage{graphicx}       % embedding of pictures
\usepackage{wrapfig} %% Obrázky v textu
\usepackage{fancyvrb}       % improved verbatim environment
\usepackage{dcolumn}        % improved alignment of table columns
\usepackage{enumitem,etoolbox,hyperref}
\usepackage{tabu}
\usepackage{multirow}
\usepackage{graphicx} %% Obrázky
\usepackage{wrapfig} %% Obrázky v textu
\usepackage[font=small,labelfont=bf]{caption} % smaller font in figure labels
\graphicspath{ {Pictures/} } %% Path to pictures directory
\usepackage[section]{placeins} % FloatBarrier around sections

\usepackage{caption} % For captions in minipages
\usepackage[normalem]{ulem} % striked letters
\usepackage{tikz} % For creating diagrams
\usetikzlibrary{calc} %for coordinate calculations in tikz
\usetikzlibrary{arrows}% in tikz graphs
\usetikzlibrary{arrows.meta}

\usepackage{float}

\usepackage[final]{listings}

\makeatletter
\AtBeginDocument{%
  \expandafter\renewcommand\expandafter\subsection\expandafter{%
    \expandafter\@fb@secFB\subsection
  }%
}
\makeatother

\newcommand{\ru}[1]{\rule{0pt}{#1 em}}%changing 
\input{mycommands.sty}

\newlength\titlebox 

\theoremstyle{definition}
\newtheorem{definition}{Definition}%[section]

\theoremstyle{plain}
\newtheorem{theorem}{Theorem}%[section]
\newtheorem*{theorem*}{Theorem}%theorem without numbering
\newtheorem{lemma}[theorem]{Lemma}

\newtheorem{proposition}[theorem]{Proposition}

\theoremstyle{remark}
\newtheorem{remark}{Remark}

\newlength{\RoundedBoxWidth}
\newsavebox{\GrayRoundedBox}
\newenvironment{GrayBox}[1][\dimexpr\textwidth-4.5ex]%
   {\setlength{\RoundedBoxWidth}{\dimexpr#1}
    \begin{lrbox}{\GrayRoundedBox}
       \begin{minipage}{\RoundedBoxWidth}}%
   {   \end{minipage}
    \end{lrbox}
    \begin{center}
    \begin{tikzpicture}%
       \draw node[draw=black,fill=black!10,rounded corners,%
             inner sep=2ex,text width=\RoundedBoxWidth]%
             {\usebox{\GrayRoundedBox}};
    \end{tikzpicture}
    \end{center}}

\newlength{\WhiteRoundedBoxWidth}
\newsavebox{\WhiteRoundedBox}
   {\setlength{\WhiteRoundedBoxWidth}{\dimexpr#1}
    \begin{lrbox}{\WhiteRoundedBox}
       \begin{minipage}{\WhiteRoundedBoxWidth}}%
   {   \end{minipage}
    \end{lrbox}
    \begin{center}
    \begin{tikzpicture}%
       \draw node[draw=black,rounded corners,%
             inner sep=2ex,text width=\WhiteRoundedBoxWidth]%
             {\usebox{\WhiteRoundedBox}};
    \end{tikzpicture}
    \end{center}}

\counterwithin{table}{section}  % Keep table numbering within sections
\counterwithin{figure}{section}  % Keep figures numbering within sections

\input{amsbibstyle.sty}
\title{The Probability that a Random Triangle in a Cube is Obtuse}
\author[ ]{Dominik Beck}
\affil[ ]{\small Faculty of Mathematics and Physics, Charles University, Prague}
\affil[ ]{\textit {\href{mailto:beckd@karlin.mff.cuni.cz}{beckd@karlin.mff.cuni.cz}}}
\date{\vspace{-2.3em}}
\date{\today}
\begin{document}
\maketitle
\begin{abstract}
In this paper, we derive the exact formula for the probability that three randomly and uniformly selected points from the interior of the unit cube form vertices of an obtuse triangle.
\end{abstract}

{\hfill
\begin{minipage}{0.875\textwidth}
\centerline{\scriptsize \textbf{Acknowledgements}}
{\scriptsize The study was supported by the Charles University, project GA UK No. 71224 and by Charles University Research Centre program No. UNCE/24/SCI/022. We would also like to acknowledge the Dual Trimester Program: ``Synergies between modern probability, geometric analysis and stochastic geometry'' organized by the Hausdorff Research in Bonn and the impact it had on this research}
\end{minipage}
\hfill}

\newpage
\tableofcontents

\newpage
\section{Introduction}
Let $K_d \subset \mathbb{R}^d$ be a $d$-dimensional convex body and let $\vect{X},\vect{Y},\vect{Z}\sim\mathsf{Unif}(K_d)$ be three random points selected from the interior of $K_d$ uniformly and independently. The objective of this paper is to study the \emph{obtusity probability} in $K_d$. That is, the probability $\eta(K_d)$ that the random triangle $\vect{X}\vect{Y}\vect{Z}$ is \emph{obtuse}.

\vspace{1em}
In two dimensions, there are several known results. Obtusity probability was first solved in a disk $\mathbb{B}_2$ by Woolhouse \cite{woolhouse1867some} as a corollary to the Silvester problem. Later, Langford \cite{langford1969probability} found $\eta(K_2)$ for $K_2$ being a general rectangle. The special case of $\eta(C_2)$, where $C_2$ is the unit square, is shown in Table \ref{tab:obtuseprob} below, alongside with other known results. It is a simple exercise to deduce $\eta(T_2)$, where $T_2$ is a general triangle. This is, however, not the content of this paper. The general result in not very illuminating either. Hence, the table only shows $\eta(T_2^*)$ for $T_2^*$ being an equilateral triangle.

\vspace{1em}
In higher dimensions, apart from the $d$-ball (Buchta and M\"{u}ller \cite{buchta1984random}), $\eta(\solK_d)$ is not known for any $K_d$ with $d \geq 3$. The only exception, as shown in this paper, is the obtusity probability in the unit cube $C_3$, the problem suggested by Finch in \cite{finch2018mathematical}. We got
\begin{equation}
\eta(C_3) \!=\! \frac{323338}{385875}\!-\!\frac{13G}{35}\!+\!\frac{4859 \pi }{62720}\!-\!\frac{73 \pi }{1680 \sqrt{2}}\!-\!\frac{\pi ^2}{105}\!+\!\frac{3\pi  \ln 2}{224}\!-\!\frac{3\pi  \ln(1\!+\!\sqrt{2})}{224} \approx 0.542659281427229,
\end{equation}
where $G= \sum_{n=0}^\infty\frac{(-1)^n}{(2n+1)^2}\approx 0.9159655941$ is the \emph{Catalan's constant}\index{Catalan's constant}. This result is new as far as we know.

\bgroup
% Additional commands control the vertical padding of the following tables:
\renewcommand{\ru}[1]{\rule{0pt}{#1 em}}%changing height in a cell
\newcommand{\tup}{\ru{1.3}} % change number to increase upper space in table 
\newcommand{\tdown}{0.5} % change number to increase down space in table
\begin{table}[h]
\centering
    \begin{tabular}{|c|c|>{\centering\arraybackslash}m{10.5ex}|>{\centering\arraybackslash}m{30ex}|}
	\hline
\multicolumn{2}{|c|}{$K_d$} & numerical value & $\eta(\solK_d)$\\[\tdown em]
    \hline \hline
\tup \small $\mathbb{B}_2$ & disk, \cite{woolhouse1867some} & $0.7197$ & $\textstyle \frac{9}{8} - \frac{4}{\pi^2}$ \\[\tdown em]
    \hline
\tup \small $\mathbb{B}_3$ & ball, \cite{buchta1984random} & $0.5286$ & $\textstyle \frac{37}{70}$ \\[\tdown em]
    \hline
\tup \small $C_2$  & square, \cite{langford1969probability}& $0.7252$ &     $\textstyle \frac{97}{150}+\frac{\pi}{40}$ \\[\tdown em]
    \hline
\tup \small $T_2^*$ & equilateral triangle & $0.7482$ & $\textstyle\frac{25}{4}+\frac{\pi }{12 \sqrt{3}}+\frac{393}{10} \ln\frac{\sqrt{3}}{2}$ \\[\tdown em]
    \hline
\ru{1.7} \small \begin{tabular}{c}\\[-0.4ex]$C_3$\end{tabular}  & \begin{tabular}{c}\\[-0.4ex]cube\end{tabular}& \begin{tabular}{c}\\[-1.2ex]
$0.5427$\end{tabular} &     
\begin{tabular}{c}
$\textstyle \frac{323338}{385875}\!-\!\frac{13G}{35}\!+\!\frac{4859 \pi }{62720}\!-\!\frac{73 \pi }{1680 \sqrt{2}}$\!\\[0.5ex]
$\textstyle -\frac{\pi ^2}{105}\!+\!\frac{3\pi  \ln 2}{224}\!-\!\frac{3\pi  \ln(1+\sqrt{2})}{224}$
\end{tabular}
\\[1.8em]
    \hline
\end{tabular}
\caption{Probability that a random triangle in $K_d$ is obtuse}
\label{tab:obtuseprob}
\end{table}
\egroup

\section{Preliminaries}

\subsection{Langford and related distributions}
Let $U,U',U''\sim \mathsf{Unif}(0,1)$ (independent), we define four random
variables
\begin{equation}
    \Lambda = (U'-U)(U''-U),\qquad
    \Sigma = (U-U')U,\qquad 
    \Xi = UU',\qquad
    \Omega = U(1-U).
\end{equation}
The equalities between $\Lambda,\Sigma,\Xi,\Omega$ with $U,U',U''$ have to be interpreted only in terms of distributions. That means, we will assume $\Lambda,\Sigma,\Xi,\Omega$ to be in fact independent. We say $\Lambda$ follows the \emph{Langford distribution} \cite{langford1969probability}. We call those variables as our papers' \emph{auxiliary random variables}. The probability density functions (PDFs) and the cumulative density functions (CDFs) of those are shown in Table \ref{tab:PDFCDF} below. Trivially, PDF of $U$ is $f_U(u)=1$ when $0<u<1$ and zero otherwise.

\begin{table}[H]
    \centering
    \begin{tabular}{|c|c|c|}
        \hline
        $X$ & $x$ & PDF: $f_X(x)=\frac{d}{dx}F_X(x)$ \hspace{5ex}
        CDF: $F_X(x) = \prob{X \leq x}$ \\
        \hline\hline
        \multirow{2}{*}[-3em]{$\Lambda$} & \multirow{2}{*}[-3em]{$\lambda$} & $f_\Lambda(\lambda) = \begin{cases}
 4 \argtanh\sqrt{1+4 \lambda}-4 \sqrt{1+4 \lambda}, & -\frac{1}{4}\leq \lambda <0, \\
 4 \sqrt{\lambda }-2 \ln \lambda -4, & \,\,\,\,\, 0< \lambda \leq 1,\\
 0, & \,\,\,\,\, \text{otherwise}
\end{cases}$ \\\cline{3-3}
& & $F_\Lambda(\lambda) = \begin{cases}
 0,& \,\,\,\,\,\lambda < -\tfrac14\\
 \frac{1}{3} (1-8 \lambda ) \sqrt{1+4 \lambda }+4 \lambda  \argtanh\sqrt{1+4 \lambda}, & -\frac{1}{4}\leq \lambda <0
   \\
 \tfrac13,& \,\,\,\,\lambda = 0,\\
 \frac{1}{3} \left(1-6 \lambda +8 \lambda ^{3/2}\right)-2 \lambda  \ln\lambda, & \,\,\,\,\,0< \lambda <1,\\
 1,& \,\,\,\,\lambda \geq 1.\\
\end{cases}$\\
        \hline
        \multirow{2}{*}[-3em]{$\Sigma$} & \multirow{2}{*}[-3em]{$\sigma$} & $f_\Sigma(\sigma) = \begin{cases}
 2\argtanh\sqrt{1+4 \sigma}, & -\frac{1}{4}\leq \sigma <0, \\
 -\tfrac12 \ln \sigma, & \,\,\,\,\, 0< \sigma \leq 1,\\
 0, & \,\,\,\,\, \text{otherwise}\end{cases}$ \\\cline{3-3}
 & & $F_\Sigma(\sigma) = \begin{cases}
 0,& \,\,\,\,\,\sigma < -\tfrac14\\
 \frac{1}{2} \sqrt{1+4\sigma}+2\sigma \argtanh\sqrt{1+4\sigma}, & -\frac{1}{4}\leq \sigma <0
   \\
 \tfrac12,& \,\,\,\,\sigma = 0,\\
 \frac{1}{2} (1+\sigma-\sigma\ln\sigma), & \,\,\,\,\,0\leq \sigma <1,\\
 1,& \,\,\,\,\sigma \geq 1.\\
\end{cases}$\\
        \hline
        $\Omega$ & $\omega$ & \begin{tabular}{c|c}
             $f_\Omega(\omega) = \begin{cases}
 \frac{2}{\sqrt{1-4\omega}}, & \,\,\,\,\,0\leq \omega <1/4,\\
 0,& \,\,\,\,\,\text{otherwise}\\
\end{cases}$ & $F_\Omega(\omega) =\begin{cases}
 0,& \,\,\,\,\,\omega \geq 0\\
 1-\sqrt{1-4\omega}, & \,\,\,\,\,0\leq \omega <1/4,\\
 1,& \,\,\,\,\,\omega \geq 1/4.\\
\end{cases}$ 
        \end{tabular}\\
        \hline
        $\Xi$ & $\xi$ & \begin{tabular}{c|c}
             $\!\!\!\!\!\!f_\Xi(\xi) = \begin{cases}
 -\ln\xi, & \,\,\,0\leq \xi <1,\\
 0,& \,\,\,\text{otherwise}\\
\end{cases}$ & $F_\Xi(\xi) = \begin{cases}
 0,& \,\,\,\,\,\xi \leq 0\\
 \xi(1-\ln\xi), & \,\,\,\,0< \xi <1,\\
 1,& \,\,\,\,\xi \geq 1.\\
\end{cases}$ 
        \end{tabular}\\
        \hline
    \end{tabular}
    \caption{PDFs and CDFs of auxiliary variables}
    \label{tab:PDFCDF}
\end{table}

\subsection{Reformulation of the problem in terms of random variables}

Consider a trivariate functional $\eta = \eta(\vect{X},\vect{Y},\vect{Z})$ being equal to one when the random triangle $\vect{X}\vect{Y}\vect{Z}$ is obtuse and zero otherwise. We shall call this functional the \emph{obtusity indicator}\index{obtusity indicator}. Taking expectation, we can write
\begin{equation}
\eta(\solK_d) = \prob{\vect{X}\vect{Y}\vect{Z}\text{ is obtuse}}= \expe{\eta(\vect{X},\vect{Y},\vect{Z})\mid \vect{X},\vect{Y},\vect{Z} \sim \mathsf{Unif}(K_d)}.
\end{equation}
Note that a triangle is obtuse when exactly one internal angle is obtuse. Hence, we can decompose the obtusity indicator almost surely as follows
\begin{equation}
\eta(\vect{X},\vect{Y},\vect{Z}) = \eta^*(\vect{X},\vect{Y},\vect{Z}) + \eta^*(\vect{Y},\vect{Z},\vect{X}) + \eta^*(\vect{Z},\vect{X},\vect{Y}),
\end{equation}
where we denoted $\eta^*(\vect{X},\vect{Y},\vect{Z})$ as the obtusity indicator that are equal to one when the obtuse angle is located at the first vertex $\vect{X}$. Furthermore, we can write out this indicator in terms of the characteristic function of a dot product as
\begin{equation}
\eta^*(\vect{X},\vect{Y},\vect{Z}) = \mathbbm{1}_{(\vect{Y}-\vect{X})^\top (\vect{Z}-\vect{X})<0}
\end{equation}
since $(\vect{Y}-\vect{X})^\top (\vect{Z}-\vect{X}) = \|\vect{Y}-\vect{X}\|\|\vect{Z}-\vect{X}\|\cos\alpha$, where $\alpha$ is the angle at vertex $\vect{X}$ of the triangle $\vect{X}\vect{Y}\vect{Z}$. Therefore,
\begin{GrayBox}
\vspace{-0.5em}
\begin{equation}\label{eq:etadecomp}
\eta(\vect{X},\vect{Y},\vect{Z}) = \mathbbm{1}_{(\vect{Y}-\vect{X})^\top (\vect{Z}-\vect{X})<0}+\mathbbm{1}_{(\vect{Z}-\vect{Y})^\top (\vect{X}-\vect{Y})<0}+\mathbbm{1}_{(\vect{X}-\vect{Z})^\top (\vect{Y}-\vect{Z})<0}.
\end{equation}
\end{GrayBox}

Taking expectation and by symmetry, we get for the obtusity probability
\begin{equation}
    \eta(K_d) = 3\, \prob{(\vect{Y}-\vect{X})^\top (\vect{Z}-\vect{X})<0\mid \vect{X},\vect{Y},\vect{Z} \sim \mathsf{Unif}(K_d)}.
\end{equation}
When $K_d = C_3$ (a cube), we can rewrite $\eta(C_3)$ in terms auxiliary variables. This is the method used by Langford \cite{langford1969probability} to deduce $\eta(C_2)$. In fact, for a $d$-cube $C_d$ in any dimension $d$, we may parametrize the random points $\vect{X},\vect{Y},\vect{Z}$ as
\begin{equation}
    \vect{X}=\sum_{i=1}^d X_i \vect{e}_i, \quad
    \vect{Y}=\sum_{i=1}^d Y_i \vect{e}_i, \quad
    \vect{Z}=\sum_{i=1}^d Z_i \vect{e}_i,
\end{equation}
where $X_i,Y_i,Z_i\sim\mathsf{Unif}(0,1), i=1,\ldots,d$. Hence,
\begin{equation}
(\vect{Y}-\vect{X})^\top (\vect{Z}-\vect{X}) = \sum_{i=1}^d (Y_i-X_i)(Z_i-X_i)    
\end{equation}
and thus, using our auxiliary random variables introduced in the previous section,
\begin{equation}\label{eq:intCd}
\eta(C_d) = 3\prob{\sum_{i=1}^d \Lambda_i < 0} = 3\int_{\lambda_1+\cdots+\lambda_d<0} f_\Lambda(\lambda_1)\ldots f_\Lambda(\lambda_d) \ddd\lambda_1 \cdots \dd\lambda_d,
\end{equation}
where $\Lambda_i,i=1,\ldots, d$ are independent random variables following the Langford distribution. In fact, this is exactly the integral that eanbled Langford to derive the $d=2$ case.

\subsection{Crofton Reduction Technique}
Unfortunately, we were not able to find the closed form expression of the integral in Equation \eqref{eq:intCd} with $d=3$ straightaway. The intermediate result involves dilogarithms with intricate arguments. However, there is a workaround -- the \emph{Crofton Reduction Technique} (CRT). Instead of finding the obtusity probability with vertices of the random triangle being selected from the interior of $C_3$, we can instead select those vertices from special combinations of faces, edges and vertices. Naturally, those problems are simpler to deduce. It is rather miraculous though, that by combining those result, we can get the solution of the original problem. This is the core idea of the CRT in its generalised form introduced by Ruben and Reed \cite{ruben1973more}. See Sullivan's thesis \cite{sullivan1996crofton} for a detailed introduction into CRT and for more examples.

\subsubsection{Definitions}
\begin{definition}
A polytope $A \subset \mathbb{R}^d$ of dimension $\dim A = a \in \{0,1,2,\ldots,d\}$ and $a-$volume $\volA$ is defined as a connected and finite union of $a$-dimensional simplices. We say a polytope is \textbf{flat}\index{flat domain} if $\dim \mathcal{A}(A) = \dim A$, where $\mathcal{A}(A)$ stands for the affine hull of $A$. Note that any polytope with $a=d$ is flat automatically.
\end{definition}

\begin{definition}
We denote $\mathcal{P}_a(\mathbb{R}^d)$ the set of flat polytopes of dimension $a$ in $\mathbb{R}^d$ and denote $\mathcal{P}(\mathbb{R}^d) = \bigcup_{0\leq a \leq d} \mathcal{P}_a(\mathbb{R}^d)$ the set of all flat polytopes in $\mathbb{R}^d$. Finally, we denote $\mathcal{P}_+(\mathbb{R}^d) = \mathcal{P}(\mathbb{R}^d)\setminus\mathcal{P}_0(\mathbb{R}^d)$ (flat polytopes excluding points).
\end{definition}

\begin{definition}
Let $A,B \in \mathcal{P}(\mathbb{R}^d)$ and $P: \mathbb{R}^d \times \mathbb{R}^d \to \mathbb{R}$, we denote the mean value $P_{AB} = \Ex \left[ P(\vect{X},\vect{Y}) \, |\, \vect{X}\sim \mathsf{Unif}(A), \vect{Y} \sim \mathsf{Unif}(B), \,\text{independent}\right]$. Whenever it is unambiguous, we write $P_{ab}$ where $a=\dim A$ and $b = \dim B$ instead of $P_{AB}$. If there is still ambiguity, we can add additional letters after as superscripts to distinguish between various mean values $P_{ab}$.
\end{definition}

\begin{proposition}
For any $A \in \mathcal{P}_a(\mathbb{R}^d)$ with $a> 0$, there exist \textbf{convex} $\partial_i A \in \mathcal{P}_{a-1}(\mathbb{R})$ (\textbf{sides} of $A$) such that $\partial A = \bigcup_i \partial_i A$ with pairwise intersection of $\partial_i A$ having $(a\!-\!1)$-volume equal to zero.
\end{proposition}
\begin{comment}
%Rataj: simpl. con. satisfied automatically due to convexity
\begin{remark}
Without loss of generality, we will always assume $\partial_i A$ are \textbf{simply connected} (if they are not, we split those to a finite union of simply connected ones).
\end{remark}
\end{comment}
\begin{remark}
The sides of three dimensional polytopes (polyhedra) are called \textbf{faces}.%(for those we also assume simple connectedness).
\end{remark}
\begin{definition}
Let $A \in \mathcal{P}_+(\mathbb{R}^d)$. Let $\uvect{n}_i$ be the outer normal unit vector of $\partial_i A$ in $\mathcal{A}(A)$, then we define a \textbf{signed distance}\index{signed distance} $h_\vect{C}(\partial_i A)$ from a given point $\vect{C} \in \mathcal{A}(A)$ to $\partial_i A$ as the dot product $\vect{v}_i^\top \uvect{n}_i $, where $\vect{v}_i = \vect{x}_i - \vect{C}$ and $x_i \in \partial_i A$ arbitrary. Note that if $A$ is convex, the signed distance coincides with the \textbf{support function}\index{support function} $h(A - \vect{C},\uvect{n}_i)$ defined for any convex domain $B$ as $h(B,\uvect{n}_i) = \sup_{\vect{b} \in B} \vect{b}^\top \uvect{n}_i $.
\end{definition}

\begin{definition}
\label{def:weights}
Let $A \in \mathcal{P}_+(\mathbb{R}^d)$ with $a = \dim A$. Even though $\partial A \notin \mathcal{P}(\mathbb{R}^d)$, we still define $P_{\partial A \, B}$ for a given point (called the scaling point\index{scaling point}) $\vect{C} \in \mathcal{A}(A)$ as a weighted mean via the relation
\begin{equation}
    P_{\partial A\, B} = \sum_i w_i P_{\partial_i A\, B} 
\end{equation}
with weights $w_i$ (may be also negative) equal to
\begin{equation}
    w_i = \frac{\vol{\partial_i A}}{a \volA} h_\vect{C}(\partial_i A).
\end{equation}
\end{definition}

\begin{definition}
Let $P: \mathbb{R}^d \times \mathbb{R}^d \to \mathbb{R}$. We say the functional $P$ is \textbf{homogeneous}\index{homogeneous functional} of order $p \in \mathbb{R}$, if there exists $\tildee{P}: \mathbb{R}^d \to \mathbb{R}$ such that $P(\vect{x},\vect{y}) = \tildee{P}(\vect{x}-\vect{y})$ for all $\vect{x},\vect{y} \in \mathbb{R}^d$ and $\tildee{P}(r v) = r^p \tildee{P}(v)$ for all $v \in \mathbb{R}^d$ and all $r > 0$. We write $p = \dim P$ (although $p$ might not be an integer). We say $P$ is \textbf{symmetric} if $P(\vect{x},\vect{y})=P(\vect{y},\vect{x})$. Finally, if $P$ is a functional of two points, we say it is \textbf{bivariate}\index{homogeneous functional!bivariate}. If it depends of more points, we say it is \textbf{multivariate}\index{homogeneous multivariate}.
\end{definition}

\subsubsection{Bivariate Crofton Reduction Technique}

\begin{lemma}[Bivariate Crofton Reduction Technique] Let $P: \mathbb{R}^d\times \mathbb{R}^d \to \mathbb{R}$ be homogeneous of order $p$ and $A,B \in \mathcal{P}(\mathbb{R}^d)$, then for any $\vect{C} \in \mathcal{A}(A)\cap \mathcal{A}(B)$ it holds\index{Crofton Reduction Technique!bivariate}
\begin{equation}
p P_{AB} = a (P_{\partial A\, B} - P_{AB}) + b (P_{A\,\partial B} - P_{AB}).
\end{equation}
\begin{figure}[h]
    \centering
     \includegraphics[height=0.18\textwidth]{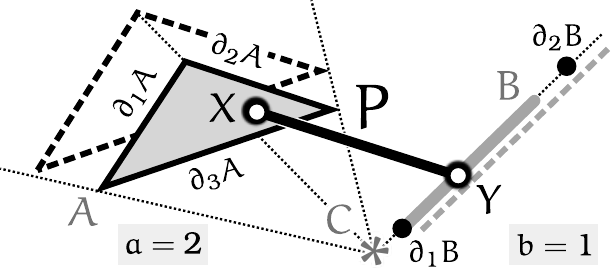}
    \caption{Bivariate Crofton Reduction Technique}
    \label{fig:Crofton}
\end{figure}
\end{lemma}
\begin{proof}
The formula is a special case of the extension of the Crofton theorem by Ruben and Reed \cite{ruben1973more}. See our own proof in our earlier work on mean distances here \cite{beckmeandist}.
\end{proof}

To find the expectation of $P$, in the first step, we choose $A=K$ and $B=K$. Since the affine hulls of both $A$ and $B$ fill the whole space $\mathbb{R}^d$, any point in $\mathbb{R}^d$ can be selected for $\vect{C}$. We then employ the reduction technique to express $P_{AB}$ in $P_{A'B'}$ where $A'$ and $B'$ have smaller dimensions then $A$ and $B$. The pairs of various $A'$ and $B'$ we encounter we call \textbf{configurations}. The process is repeated until the affine hull intersection of $A'$ and $B'$ is empty. In that case, we have reached an \textbf{irreducible} configuration.

\subsubsection{Multivariate Crofton Reduction Technique}
Let us instead consider multivariate functionals $P$ (dependent on more that only two points). One example is area, volume or obtusity. CRT naturally generalises.
\begin{definition}
Let $P=P(x_1,x_2,\ldots,x_n)$ be a homogenous function of $n$ points.
We define $P_{A_1 A_2\ldots A_n} = \Ex \left[ P(\vect{X}_1, \ldots ,\vect{X}_n) \, |\, \vect{X}_1\sim \mathsf{Unif}(A_1),\ldots, \vect{X}_n\sim \mathsf{Unif}(A_n)\right]$, where $A_j,j=1,\ldots,n$ are flat domains from which the points $\vect{X}_j$ are selected randomly uniformly (according to distribution $\mathsf{Unif}(A_j)$).
\end{definition}

\begin{lemma}[Multivariate Crofton Reduction Technique] Let $P: (\mathbb{R}^d)^n \to \mathbb{R}$ be homogeneous of order $p$ and $A_1,\ldots,A_n \in \mathcal{P}(\mathbb{R}^d)$, $a_i = \dim A_i$, then for any $\vect{C} \in \bigcap_{1\leq i \leq n} \mathcal{A}(A_i)$ (scaling point) it holds\index{Crofton Reduction Technique!multivariate}
\begin{equation}\label{eq:CRTmulti}
\begin{split}
p P_{A_1 A_2 \ldots A_n} & = a_1 (P_{\partial A_1\, A_2 \ldots A_n} - P_{A_1 \ldots A_n}) + a_2 (P_{A_1\partial A_2\ldots A_n} - P_{A_1\ldots A_n})\\
& + \cdots + a_n (P_{A_1 A_2 \ldots \partial A_n} - P_{A_1 \ldots A_n}).
\end{split}
\end{equation}
\end{lemma}
\begin{remark}
Symmetry of in points $\vect{X}_1,\ldots,\vect{X}_n$ is not required for CRT to hold. However, we often assume so. As a result, $P_{A_1,\ldots,A_n}$ is invariant with respect to permutations of $A_1,\ldots,A_n$.
\end{remark}

\subsection{Reduction of a Cube}

\subsubsection{Configurations}
Consider a trivariate symmetric homogeneous functional $P$ of order $p$ dependent on three random points picked uniformly from the unit cube $C_3$ with volume $\vol_3 C_3 = 1$ and let $P_{abc} = \expe{P(\vect{X},\vect{Y},\vect{Z})\mid \vect{X} \sim \mathsf{Unif}(A),\vect{Y}\sim \mathsf{Unif}(B), \vect{Z}\sim \mathsf{Unif}(C)}$, where $a = \dim A$, $b = \dim B$, $c = \dim C$ and the concrete selection of $A,B,C$ is deduced from the reduction diagram in Figure \ref{fig:TrianglePickingCube} below.

\begin{figure}[h]
    \centering
     \includegraphics[width=0.98\textwidth]{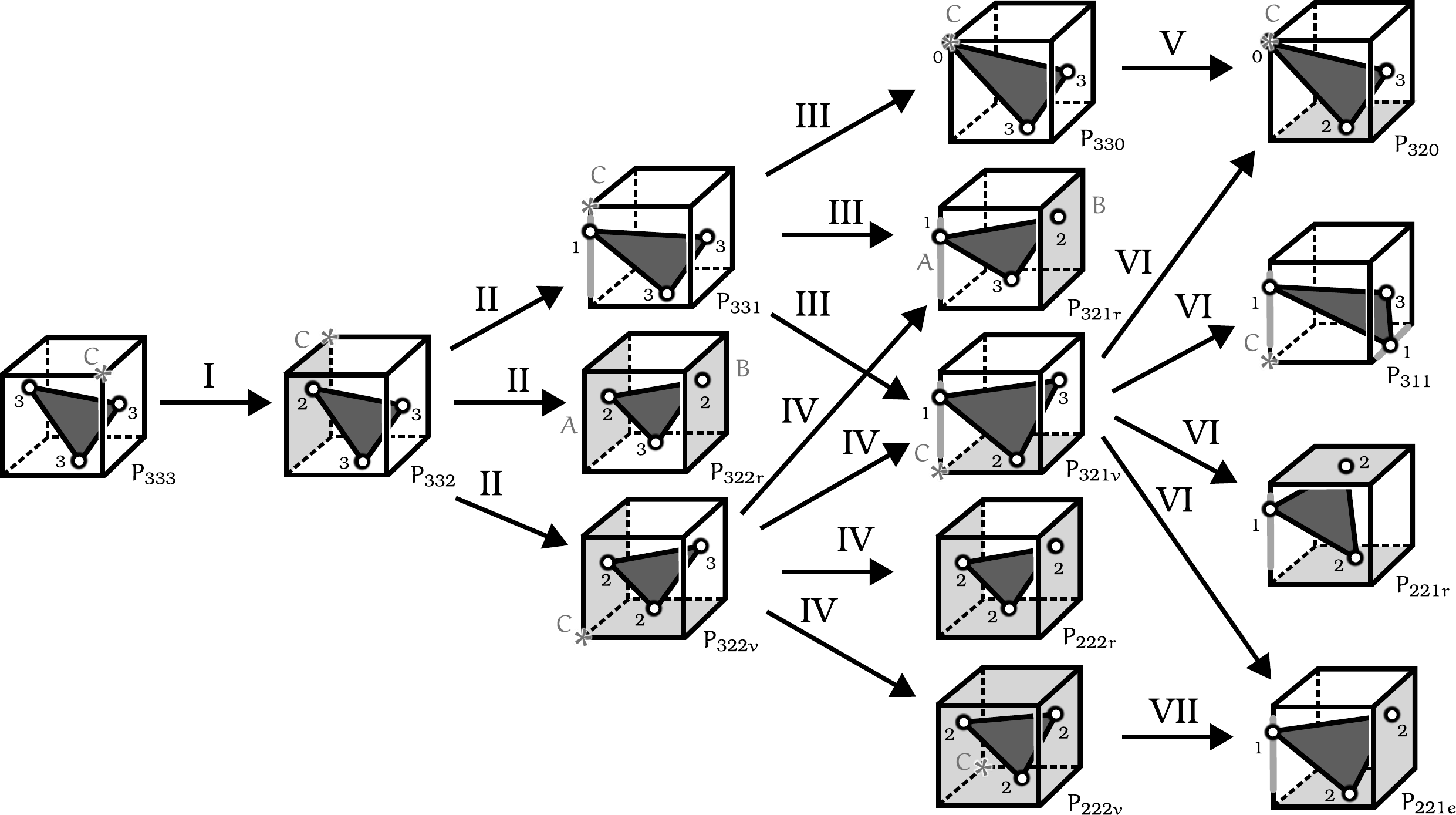}
    \caption{All different $P_{abc}$ sub-configurations in $C_3$}
    \label{fig:TrianglePickingCube}
\end{figure}

\subsubsection{Reduction system}
The full system obtained by CRT is
\begin{align*}
    \mathbf{I} & :\, p P_{333} \,\,= 3\cdot 3 (P_{332} - P_{333})\\
    \mathbf{II} & :\, p P_{332} \,\,= 2(P_{331}-P_{332})+2\cdot3(P_{322}-P_{332}),\\
    \mathbf{III} & :\, p P_{331} \,\,= 1(P_{330}-P_{331}) + 2\cdot3 (P_{321} - P_{331}),\\
    \mathbf{IV} & :\, p P_{322v} = 2\cdot2 (P_{321}' - P_{322v}) + 3(P_{222}-P_{322v}),\\
    \mathbf{V} & :\, p P_{330} \,\,= 2\cdot 3(P_{320}-P_{330}),\\
    \mathbf{VI} & :\, p P_{321v} = 1(P_{320}-P_{321v})+2(P_{311}-P_{321v})+3(P_{221}-P_{321v}),\\
    \mathbf{VII} & :\, p P_{222v} = 3\cdot2(P_{221e}-P_{222v})\\
\end{align*}
with
\begin{align*}
    P_{322} & = \tfrac13 P_{322r}+\tfrac23 P_{322v},\\
    P_{321} & = \tfrac23 P_{321r}+\tfrac13 P_{321v},\\
    P_{321}' & = \tfrac12 P_{321r}+\tfrac12 P_{321v},\\
    P_{222} & = \tfrac23 P_{222r}+\tfrac13 P_{222v},\\
    P_{221} & = \tfrac13 P_{221r}+\tfrac23 P_{221e}.\\
\end{align*}
The solution of our system is
\begin{equation}\label{Eq:CRTcube}
    P_{333} = \frac{108 (4 P_{221e}+P_{221r}+2 P_{311}+2 P_{320})}{(6+p) (7+p) (8+p) (9+p)}+\frac{72 (P_{222r}+2 P_{321r})}{(7+p) (8+p)
   (9+p)}+\frac{18 P_{322r}}{(8+p) (9+p)}.
\end{equation}
Especially, for functionals of order $p=0$, we get
\begin{equation}
P_{333}  = \frac{1}{28} (4 P_{221e}+P_{221r}+4 P_{222r}+2 P_{311}+2 P_{320}+8 P_{321r}+7 P_{322r}).
\end{equation}

\newpage
\section{Obtusity probability as a reducible functional}

We can use CRT to reduce the obtusity probability $\eta(\solK_d)$. Note that $\eta = \eta(\vect{X},\vect{Y},\vect{Z})$ is symmetric, trivarite and homogeneous of order zero. In our convention,
\begin{equation}
\eta(\solK_d)=\eta_{\solK_d\solK_d\solK_d} = \eta_{ddd} = \expe{\eta(\vect{X},\vect{Y},\vect{Z})\mid \vect{X},\vect{Y},\vect{Z} \sim \mathsf{Unif}(K_d)}.
\end{equation}
In a given configuration $\vect{X}\sim\mathsf{Unif}(A),\vect{Y}\sim\mathsf{Unif}(B), \vect{Z}\sim\mathsf{Unif}(C)$, we write
\begin{equation}
    \eta_{ABC} = \expe{\eta(\vect{X},\vect{Y},\vect{Z})\mid \vect{X}\sim\mathsf{Unif}(A),\vect{Y}\sim\mathsf{Unif}(B), \vect{Z}\sim\mathsf{Unif}(C)}.
\end{equation}
Additionally, we indicate by ${}^*$ the position of the obtuse vertex, so
\begin{equation}
\begin{split}
& \eta_{A^*BC} = \expe{\eta^*(\vect{X},\vect{Y},\vect{Z}) \mid \vect{X}\sim\mathsf{Unif}(A),\vect{Y}\sim\mathsf{Unif}(B), \vect{Z}\sim\mathsf{Unif}(C)}\\
& = \prob{ (\vect{Y}-\vect{X})^\top (\vect{Z}-\vect{X})<0\mid \vect{X}\sim\mathsf{Unif}(A),\vect{Y}\sim\mathsf{Unif}(B), \vect{Z}\sim\mathsf{Unif}(C)},
\end{split}
\end{equation}
similarly for $\eta_{AB^*C}$ and $ \eta_{ABC^*}$. Hence, we may write the expected value of the obtusity indicator in any configuration as
\begin{equation}
    \eta_{ABC} = \eta_{A^*BC} + \eta_{AB^*C} + \eta_{ABC^*}.
\end{equation}

\phantomsection
\subsection{\texorpdfstring{$\bm{\eta}$\textsubscript{322r}}{eta322r}}
%\addcontentsline{toc}{subsubsection}{eta322r}
In configuration $(332r)$, the first vertex $\vect{X}$ of the inscribed random triangle $\vect{X}\vect{Y}\vect{Z}$ is selected from the interior of $C_3$, while the other two $\vect{Y}$ and $\vect{Z}$ are picked from (any fixed) opposite faces. We may parametrize the points as
\begin{equation}
    \vect{X}=X_1\vect{e}_1+X_2\vect{e}_2+X_3\vect{e}_3, \quad \vect{Y}=Y_1\vect{e}_1+Y_2\vect{e}_2, \quad \vect{Z}=Z_1\vect{e}_1+Z_2\vect{e}_2+\vect{e}_3,
\end{equation}
where $X_1,X_2,X_3,Y_1,Y_2,Z_1,Z_2\sim\mathsf{Unif}(0,1)$. Based on the exact location of the obtuse angle, we recognize three sub-configurations $(3^*22r)$, $(32^*2r)$ and $(322^*r)$, out of which the last two give the same contribution by symmetry. Expressing the dot products in the decomposition of the obtusity indicator (Equation \eqref{eq:etadecomp}), we get
\begin{equation}
\begin{split}
(3^*22r)& : (\vect{Y}\!-\!\vect{X})^\top (\vect{Z}\!-\!\vect{X}) \!=\! (Y_1\!-\!X_1)(Z_1\!-\!X_1)\!+\!(Y_2\!-\!X_2)(Z_2\!-\!X_2)\!-\!X_3(1\!-\!X_3),\\
(32^*2r)& : (\vect{Z}\!-\!\vect{Y})^\top (\vect{X}\!-\!\vect{Y}) \!=\! (X_1\!-\!Y_1)(Z_1\!-\!Y_1)\!+\!(X_2\!-\!Y_2)(Z_2\!-\!Y_2) \!+\! X_3.
\end{split}
\end{equation}
The obtusity probability in $(3^*22r)$ sub-configuration in terms of auxiliary random variables is given as
\begin{equation}
\begin{split}
\eta_{3^*22r} & =\prob{(\vect{Y}-\vect{X})^\top (\vect{Z}-\vect{X})<0}=\prob{\Lambda+\Lambda'-\Omega<0} \\
& = \int_0^{1/4}\int_{-1/4}^{\omega+1/4} \int_{-1/4}^{\omega-\lambda} f_\Lambda(\lambda) f_\Lambda(\lambda') f_\Omega(\omega) \ddd\lambda'\dd\lambda\dd \omega \\
& = \int_0^{1/4}\int_{-1/4}^{\omega+1/4} f_\Lambda(\lambda) F_\Lambda(\omega-\lambda) f_\Omega(\omega) \ddd \lambda\dd \omega.  
\end{split}
\end{equation}
Unfortunately, the leftover integral is far from trivial and even \emph{Mathematica} is unable to find its closed form solution straightaway. Nevertheless, via simple Weierstrass substitution, the integral can be decomposed into linear combination of special integrals recently discussed on \emph{MSE} website \cite{MSEacoth,MSEatanh}, via which
\begin{equation}
    \eta_{3^*22r} \!=\!\frac{6739}{6750}-\frac{8 G}{15}+\frac{211 \pi }{1440}-\frac{17 \pi }{252 \sqrt{2}}-\frac{\pi ^2}{45}-\frac{\pi \ln(1\!+\!\sqrt2)}{24}+\frac{\pi \ln 2}{24} \approx 0.576363509,
\end{equation}
where $G= \sum_{n=0}^\infty\frac{(-1)^n}{(2n+1)^2}\approx 0.9159655941$ is the \emph{Catalan's constant}\index{Catalan's constant}. Somehow, the situation is much more elementary in $(32^*2r)$ configuration. With PDF of $U$ being $f_{U}(u) = \mathbbm{1}_{u\in(0,1)}$, we have
\begin{equation}
\begin{split}
\eta_{32^*2r} & =\prob{(\vect{X}-\vect{Y})^\top (\vect{Z}-\vect{Y})<0} =\prob{\Lambda+\Lambda'+U<0} \\
& = \int_{-1/4}^{1/4}\int_{-1/4}^{-\lambda}\int_0^{-\lambda-\lambda'} f_\Lambda(\lambda) f_\Lambda(\lambda') f_U(u) \ddd u \dd \lambda' \dd \lambda \\
& = \int_{-1/4}^{1/4} F_\Lambda(-\lambda) F_\Lambda(\lambda) \ddd \lambda = \frac{121}{7350}+\frac{\pi }{2688} \approx 0.0176313323.
\end{split}
\end{equation}
Lastly, by symmetry, $\eta_{322^*r}=\eta_{32^*2r}$. Summing up the three obtusity probabilities,
\begin{equation}
\begin{split}
     &\eta_{322r} = \eta_{3^*22r} + \eta_{32^*2r} + \eta_{322^*r} = \eta_{3^*22r} + 2\eta_{32^*2r} = \frac{341101}{330750}-\frac{8 G}{15}+\frac{2969 \pi }{20160} \\
    &
    -\frac{17 \pi }{252 \sqrt{2}}-\frac{\pi^2}{45}+\frac{\pi\ln 2}{24}-\frac{\pi \ln(1+\sqrt{2})}{24}  \approx 0.611626173665235356686.
\end{split}
\end{equation}

\phantomsection
\subsection{\texorpdfstring{$\bm{\eta}$\textsubscript{321r}}{eta321r}}
%\addcontentsline{toc}{subsubsection}{eta321r}
In configuration $(321r)$, vertex $\vect{X}$ is selected from the interior of $C_3$ and $\vect{Y}$ and $\vect{Z}$ are picked from one face and its opposite edge, respectively. We may parametrize the points as
\begin{equation}
    \vect{X}=X_1\vect{e}_1+X_2\vect{e}_2+X_3\vect{e}_3, \quad \vect{Y}=Y_1\vect{e}_1+Y_2\vect{e}_2, \quad \vect{Z}=Z_1\vect{e}_1+\vect{e}_3,
\end{equation}
where $X_1,X_2,X_3,Y_1,Y_2,Z_1\sim\mathsf{Unif}(0,1)$. Based on the exact location of the obtuse angle, we recognize three sub-configurations $(3^*21r)$, $(32^*1r)$ and $(321^*r)$. Expressing the dot products in the decomposition of the obtusity indicator (Equation \eqref{eq:etadecomp}), we get
\begin{equation}
\begin{split}
(3^*21r)& : (\vect{Y}\!-\!\vect{X})^\top (\vect{Z}\!-\!\vect{X}) \!=\! (Y_1\!-\!X_1)(Z_1\!-\!X_1)\!+\!(X_2\!-\!Y_2)X_2-\!X_3(1\!-\!X_3),\\
(32^*1r)& : (\vect{Z}\!-\!\vect{Y})^\top (\vect{X}\!-\!\vect{Y}) \!=\! (X_1\!-\!Y_1)(Z_1\!-\!Y_1)\!+\!(Y_2\!-\!X_2)Y_2 \!+\! X_3,\\
(321^*r)& : (\vect{X}\!-\!\vect{Z})^\top (\vect{Y}\!-\!\vect{Z}) \!=\! (X_1\!-\!Z_1)(Y_1\!-\!Z_1)+X_2Y_2+1- X_3.
\end{split}
\end{equation}
Using the auxiliary random variables, we can write the obtusity probability in $(3^*21r)$ sub-configuration as
\begin{equation}
\begin{split}
\eta_{3^*21r} & =\prob{(\vect{Y}-\vect{X})^\top (\vect{Z}-\vect{X})<0}=\prob{\Lambda+\Sigma-\Omega<0} \\
& = \int_0^{1/4}\int_{-1/4}^{\omega+1/4} \int_{-1/4}^{\omega-\lambda} f_\Lambda(\lambda) f_\Sigma(\sigma) f_\Omega(\omega) \ddd\sigma\dd\lambda\dd \omega \\
& = \int_0^{1/4}\int_{-1/4}^{\omega+1/4} f_\Lambda(\lambda) F_\Sigma(\omega-\lambda) f_\Omega(\omega) \ddd \lambda\dd \omega.  
\end{split}
\end{equation}
By using the \emph{MSE} integrals, we get
\begin{equation}
    \eta_{3^*21r} =\frac{49043}{54000}-\frac{8 G}{15}+\frac{1567 \pi }{11520}-\frac{67 \pi }{720 \sqrt{2}}-\frac{\pi ^2}{240}+\frac{\pi  \ln 2}{192}-\frac{\pi\ln (1+\sqrt2)}{96} \approx 0.5816795685,
\end{equation}
Next, in $(32^*1r)$ configuration, we have
\begin{equation}
\begin{split}
\eta_{32^*1r} & =\prob{(\vect{X}-\vect{Y})^\top (\vect{Z}-\vect{Y})<0} =\prob{\Lambda+\Sigma+U<0} \\
& = \int_{-1/4}^{1/4}\int_{-1/4}^{-\lambda}\int_0^{-\lambda-\sigma} f_\Lambda(\lambda) f_\Sigma(\sigma) f_U(u) \ddd u \dd \sigma \dd \lambda \\
& = \int_{-1/4}^{1/4} F_\Lambda(-\lambda) F_\Sigma(\lambda) \ddd \lambda = \frac{37}{1176}+\frac{\pi }{1344} \approx 0.03380008.
\end{split}
\end{equation}
At last, in $(321^*r)$, configuration, since $1-X_3\sim\mathsf{Unif}(0,1)$, we get
\begin{equation}
\begin{split}
\eta_{321^*r} & =\prob{(\vect{X}-\vect{Z})^\top (\vect{Y}-\vect{Z})<0} =\prob{\Lambda+\Xi+U<0} \\
& = \int_{-1/4}^{0}\int_0^{-\lambda}\int_0^{-\lambda-\xi} f_\Lambda(\lambda) f_\Xi(\xi) f_U(u) \ddd u \dd \xi \dd \lambda \\
& = \int_0^{1/4} F_\Lambda(-\lambda) F_\Xi(\lambda) \ddd \lambda = \frac{43}{14700} \approx 0.00292517.
\end{split}
\end{equation}
Summing up the three obtusity probabilities we obtained in all sub-configurations,
\begin{equation}
\begin{split}
    \eta_{321r} & = \eta_{3^*21r} + \eta_{32^*1r} + \eta_{321^*r} =
    \frac{2494097}{2646000}-\frac{8 G}{15}+\frac{11029 \pi }{80640}-\frac{67 \pi }{720 \sqrt{2}}\\
    &-\frac{\pi ^2}{240}+\frac{\pi \ln 2}{192} -\frac{\pi  \ln(1+\sqrt{2})}{96}   \approx 0.61840481814327429018.
\end{split}
\end{equation}

\phantomsection
\subsection{\texorpdfstring{$\bm{\eta}$\textsubscript{222r}}{eta222r}}
%\addcontentsline{toc}{subsubsection}{eta222r}
In configuration $(222r)$, vertices $\vect{Y}$ and $\vect{Z}$ are selected from opposite faces of $C_3$ and $\vect{X}$ is selected from another face in between the two. We may parametrize the points as
\begin{equation}
    \vect{X}=X_2\vect{e}_2+X_3\vect{e}_3, \quad \vect{Y}=Y_1\vect{e}_1+Y_2\vect{e}_2, \quad \vect{Z}=Z_1\vect{e}_1+Z_2\vect{e}_2+\vect{e}_3,
\end{equation}
where $X_2,X_3,Y_1,Y_2,Z_1,Z_2\sim\mathsf{Unif}(0,1)$. Based on the exact location of the obtuse angle, we recognize three sub-configurations $(2^*22r)$, $(22^*2r)$ and $(222^*r)$, the last last two of which give the same contribution by symmetry. Expressing the dot products in the decomposition of the obtusity indicator, we get
\begin{equation}
\begin{split}
(2^*22r)& : (\vect{Y}\!-\!\vect{X})^\top (\vect{Z}\!-\!\vect{X}) \!=\!Y_1Z_1\!+\! (Y_2\!-\!X_2)(Z_2\!-\!X_2)\!-\!X_3(1\!-\!X_3),\\
(22^*2r)& : (\vect{Z}\!-\!\vect{Y})^\top (\vect{X}\!-\!\vect{Y}) \!=\! (Y_1\!-\!Z_1)Y_1\!+\!(X_2\!-\!Y_2)(Z_2\!-\!Y_2) \!+\! X_3.
\end{split}
\end{equation}
Using auxiliary random variables, we can write the obtusity probability in $(2^*22r)$ sub-configuration as
\begin{equation}
\begin{split}
\eta_{2^*22r} & =\prob{(\vect{Y}-\vect{X})^\top (\vect{Z}-\vect{X})<0}=\prob{\Xi+\Lambda-\Omega<0} \\
& = \int_0^{1/4}\int_{-1/4}^{\omega} \int_0^{\omega-\lambda} f_\Xi(\xi) f_\Lambda(\lambda) f_\Omega(\omega) \ddd\xi\dd\lambda\dd \omega \\
& = \int_0^{1/4}\int_{-1/4}^{\omega} F_\Xi(\omega-\lambda) f_\Lambda(\lambda) f_\Omega(\omega) \ddd \lambda\dd \omega.  
\end{split}
\end{equation}
By using the \emph{MSE} integrals, we get
\begin{equation}
    \eta_{2^*22r} =\frac{14393}{27000}-\frac{2 G}{15}+\frac{11 \pi }{1152}-\frac{\pi ^2}{72}+\frac{\pi\ln 2}{96}\approx 0.326548524.
\end{equation}
Next, in $(22^*2r)$ configuration,
\begin{equation}
\begin{split}
\eta_{22^*2r} & =\prob{(\vect{X}-\vect{Y})^\top (\vect{Z}-\vect{Y})<0} =\prob{\Sigma+\Lambda+U<0} = \eta_{32^*1r} \\
& = \frac{37}{1176}+\frac{\pi }{1344} \approx 0.03380008.
\end{split}
\end{equation}
At last, $\eta_{222^*r}=\eta_{22^*2r}$ by symmetry. Summing up the three obtusity probabilities,
\begin{equation}
\begin{split}
     & \eta_{222r} = \eta_{2^*22r} + \eta_{22^*2r} + \eta_{222^*r} = \eta_{2^*22r} + 2\eta_{22^*2r}
      \\
    & = \frac{788507}{1323000}-\frac{2 G}{15}+\frac{89 \pi }{8064}-\frac{\pi ^2}{72}+\frac{\pi  \ln 2}{96} \approx 0.39414868337494.
\end{split}
\end{equation}

\phantomsection
\subsection{\texorpdfstring{$\bm{\eta}$\textsubscript{320}}{eta320}}
%\addcontentsline{toc}{subsubsection}{eta320}
In configuration $(320)$, $\vect{X}$ is selected from the interior of $C_3$, $\vect{Y}$ is selected from a face and $\vect{Z}$ in one of the vertices opposite to the selected face. We may parametrize the points as
\begin{equation}
    \vect{X}=X_1\vect{e}_1+X_2\vect{e}_2+X_3\vect{e}_3, \quad \vect{Y}=Y_1\vect{e}_1+Y_2\vect{e}_2, \quad \vect{Z}=\vect{e}_3,
\end{equation}
where $X_1,X_2,X_3,Y_1,Y_2\sim\mathsf{Unif}(0,1)$. Based on the exact location of the obtuse angle, we recognize three sub-configurations $(3^*20)$, $(32^*0)$ and $(320^*)$. Expressing the dot products in the decomposition of the obtusity indicator, we get
\begin{equation}
\begin{split}
(3^*20)& : (\vect{Y}\!-\!\vect{X})^\top (\vect{Z}\!-\!\vect{X}) \!=\!(X_1-Y_1)X_1+(X_2-Y_2)X_2-X_3(1-X_3),\\
(32^*0)& : (\vect{Z}\!-\!\vect{Y})^\top (\vect{X}\!-\!\vect{Y}) \!=\! (Y_1-X_1)Y_1+(Y_2-X_2)Y_2 + X_3,\\
(320^*)& : (\vect{X}\!-\!\vect{Z})^\top (\vect{Y}\!-\!\vect{Z}) \!=\! X_1Y_1+X_2Y_2+1-X_3,\\
\end{split}
\end{equation}
Using auxiliary random variables, we can write the obtusity probability in $(3^*20)$ sub-configuration as
\begin{equation}
\begin{split}
\eta_{3^*20} & =\prob{(\vect{Y}-\vect{X})^\top (\vect{Z}-\vect{X})<0}=\prob{\Sigma+\Sigma'-\Omega<0} \\
& = \int_0^{1/4}\int_{-1/4}^{\omega+1/4} \int_{-1/4}^{\omega-\sigma} f_\Sigma(\sigma) f_\Sigma(\sigma') f_\Omega(\omega) \ddd\sigma'\dd\sigma\dd \omega \\
& = \int_0^{1/4}\int_{-1/4}^{\omega+1/4} f_\Sigma(\sigma) F_\Sigma(\omega-\sigma) f_\Omega(\omega) \ddd \sigma\dd \omega.  
\end{split}
\end{equation}
By using the \emph{MSE} integrals, we get
\begin{equation}
    \eta_{3^*20} =\frac{42977}{54000}-\frac{7G}{30}-\frac{\pi ^2}{1440}\approx 0.575291173117.
\end{equation}
Next, in $(32^*0)$ configuration,
\begin{equation}
\begin{split}
\eta_{32^*0} & =\prob{(\vect{X}-\vect{Y})^\top (\vect{Z}-\vect{Y})<0} =\prob{\Sigma+\Sigma'+U<0} \\
& = \int_{-1/4}^{1/4}\int_{-1/4}^{-\sigma}\int_0^{-\sigma-\sigma'} f_\Sigma(\sigma) f_\Sigma(\sigma') f_U(u) \ddd u \dd \sigma' \dd \sigma \\
& = \int_{-1/4}^{1/4} F_\Sigma(-\sigma) F_\Sigma(\sigma) \ddd \sigma = \frac{23}{450} \approx 0.0511111.
\end{split}
\end{equation}
At last, in $(320^*)$, configuration, since $1-X_3\sim\mathsf{Unif}(0,1)$, we get trivially
\begin{equation}
\eta_{320^*} =\prob{(\vect{X}-\vect{Z})^\top (\vect{Y}-\vect{Z})<0} =\prob{\Xi+\Xi'+U<0} =0.
\end{equation}
Summing up the three obtusity probabilities we got,
\begin{equation}
\begin{split}
    \eta_{320} & = \eta_{3^*20} + \eta_{32^*0} + \eta_{320^*} =
    \frac{45737}{54000}-\frac{7G}{30}-\frac{\pi ^2}{1440} \approx 0.6264022842.
\end{split}
\end{equation}

\phantomsection
\subsection{\texorpdfstring{$\bm{\eta}$\textsubscript{311}}{eta311}}
%\addcontentsline{toc}{subsubsection}{eta311}
In configuration $(311)$, $\vect{X}$ is selected from the interior of $C_3$ and $\vect{Y}$ and $\vect{Z}$ are selected from perpendicular edges which do not share a common vertex. We may parametrize the points as
\begin{equation}
    \vect{X}=X_1\vect{e}_1+X_2\vect{e}_2+X_3\vect{e}_3, \quad \vect{Y}=Y_3\vect{e}_3, \quad \vect{Z}=\vect{e}_1+Z_2\vect{e}_2,
\end{equation}
where $X_1,X_2,X_3,Y_3,Z_2\sim\mathsf{Unif}(0,1)$. Based on the exact location of the obtuse angle, we recognize three sub-configurations $(3^*11)$, $(31^*1)$ and $(311^*)$, out of which the last two give the same contribution. Expressing the dot products in the decomposition of the obtusity indicator, we get
\begin{equation}
\begin{split}
(3^*11)& : (\vect{Y}\!-\!\vect{X})^\top (\vect{Z}\!-\!\vect{X}) =-(1\!-\!X_1)X_1+(X_2\!-\!Z_2)X_2-(X_3\!-\!Y_3)X_3,\\
(31^*1)& : (\vect{Z}\!-\!\vect{Y})^\top (\vect{X}\!-\!\vect{Y}) = X_1+X_2Z_2+(Y_3-X_3)Y_3,\\
\end{split}
\end{equation}
Using auxiliary random variables, we can write the obtusity probability in $(3^*20)$ sub-configuration as
\begin{equation}
\begin{split}
\eta_{3^*11} & =\prob{(\vect{Y}-\vect{X})^\top (\vect{Z}-\vect{X})<0}=\prob{-\Omega+\Sigma+\Sigma'<0} = \eta_{3^*20} \\
& = \frac{42977}{54000}-\frac{7G}{30}-\frac{\pi ^2}{1440}\approx 0.575291173117.  
\end{split}
\end{equation}
Next, in $(31^*1)$ configuration with PDF of $U$ being $f_{U}(u) = \mathbbm{1}_{u\in(0,1)}$, we have
\begin{equation}
\begin{split}
\eta_{31^*1} & =\prob{(\vect{X}-\vect{Y})^\top (\vect{Z}-\vect{Y})<0} =\prob{U+\Xi+\Sigma<0} \\
& = \int_{-1/4}^{0}\int_0^{-\sigma}\int_0^{-\sigma-\xi} f_U(u) f_\Xi(\xi) f_\Sigma(\sigma) \ddd u \dd \xi \dd \sigma \\
& = \int_0^{1/4} F_\Sigma(-\sigma) F_\Xi(\sigma) \ddd \sigma = \frac{17}{1800} \approx 0.00944444.
\end{split}
\end{equation}
At last, by symmetry, $\eta_{311^*}=\eta_{31^*1}$. Summing up the three obtusity probabilities,
\begin{equation}
\begin{split}
    \eta_{320} & = \eta_{3^*20} + \eta_{32^*0} + \eta_{320^*} =
    \frac{43997}{54000}-\frac{7 G}{30}-\frac{\pi ^2}{1440} \approx 0.5941800620.
\end{split}
\end{equation}

\phantomsection
\subsection{\texorpdfstring{$\bm{\eta}$\textsubscript{221r}}{eta221r}}
%\addcontentsline{toc}{subsubsection}{eta221r}
In configuration $(221r)$, $\vect{X}$ and $\vect{Y}$ are selected from opposite faces of $C_3$ while $\vect{Z}$ is selected from one of the edges connecting them. We may parametrize the points as
\begin{equation}
    \vect{X}=X_1\vect{e}_1+X_2\vect{e}_2, \quad \vect{Y}=Y_1\vect{e}_1+Y_2\vect{e}_2+\vect{e}_3, \quad \vect{Z}=Z_3\vect{e}_3,
\end{equation}
where $X_1,X_2,Y_1,Y_2,Z_3\sim\mathsf{Unif}(0,1)$. Based on the exact location of the obtuse angle, we recognize three sub-configurations $(2^*21r)$, $(22^*1r)$ and $(221^*r)$, out of which the first two give the same contribution by symmetry. Expressing the dot products in the decomposition of the obtusity indicator, we get
\begin{equation}
\begin{split}
(2^*21r)& : (\vect{Y}\!-\!\vect{X})^\top (\vect{Z}\!-\!\vect{X}) =(X_1\!-\!Y_1)X_1+(X_2\!-\!Y_2)X_2+Z_3,\\
(221^*r)& : (\vect{X}\!-\!\vect{Z})^\top (\vect{Y}\!-\!\vect{Z}) = X_1Y_1+X_2Y_2+Z_3(1-Z_3),\\
\end{split}
\end{equation}
Using auxiliary random variables, we can write the obtusity probability in $(2^*21r)$ sub-configuration as
\begin{equation}
\eta_{2^*21r} =\prob{(\vect{Y}\!-\!\vect{X})^\top (\vect{Z}\!-\!\vect{X})<0}=\prob{\Sigma+\Sigma'-U<0} = \eta_{32^*0} = \frac{23}{450}.
\end{equation}
By symmetry, $\eta_{22^*1r} = \eta_{2^*21r}$. Finally, in $(221^*r)$ configuration,
\begin{equation}
\begin{split}
& \eta_{221^*r} =\prob{(\vect{X}-\vect{Z})^\top (\vect{Y}-\vect{Z})<0} =\prob{\Xi+\Xi'-\Omega<0} \\
& = \int_0^{\frac14}\!\int_0^{\frac14-\xi}\!\!\int_{\xi+\xi'}^{\frac14} f_\Xi(\xi) f_\Xi(\xi') f_\Sigma(\sigma) \dd \sigma\dd \xi' \dd \xi \!=\!\! \int_0^{\frac14}\!\!\int_0^{\frac14-\xi} \!\ln\xi\ln\xi' \sqrt{1\!-\!4(\xi\!+\!\xi')} \ddd\xi'\dd\xi \\
& = \!\int_0^{\frac14} \frac{(1\!-\!4 \xi )^{3/2}}{18} (3 \ln (1\!-\!4 \xi )\!-\!8) \ln \xi \ddd\xi = \frac{788}{3375}\!-\!\frac{\pi ^2}{120} \approx 0.151234778.
\end{split}
\end{equation}
Summing up the three obtusity probabilities,
\begin{equation}
\begin{split}
    \eta_{221r} & = \eta_{2^*21r} + \eta_{22^*1r} + \eta_{221^*r} =
    \frac{1133}{3375}-\frac{\pi ^2}{120} \approx 0.2534570004.
\end{split}
\end{equation}

\phantomsection
\subsection{\texorpdfstring{$\bm{\eta}$\textsubscript{221e}}{eta221e}}
%\addcontentsline{toc}{subsubsection}{eta221e}
In the last irreducible configuration $(221e)$, $\vect{X}$ and $\vect{Y}$ are selected from adjacent faces of $C_3$ while $\vect{Z}$ is selected from an edge opposite to the face on which $\vect{Y}$ reside. We may parametrize the points as
\begin{equation}
    \vect{X}=X_1\vect{e}_1+X_2\vect{e}_2, \quad \vect{Y}=\vect{e}_1+Y_2\vect{e}_2+Y_3\vect{e}_3, \quad \vect{Z}=Z_3\vect{e}_3,
\end{equation}
where $X_1,X_2,Y_2,Y_3,Z_3\sim\mathsf{Unif}(0,1)$. Based on the exact location of the obtuse angle, we recognize three sub-configurations $(2^*21e)$, $(22^*1e)$ and $(221^*e)$. Expressing the dot products in the decomposition of the obtusity indicator, we get
\begin{equation}
\begin{split}
(2^*21e)& : (\vect{Y}\!-\!\vect{X})^\top (\vect{Z}\!-\!\vect{X}) =-X_1(1-X_1)+(X_2-Y_2)X_2+Y_3Z_3,\\
(22^*1e)& : (\vect{X}\!-\!\vect{Y})^\top (\vect{Z}\!-\!\vect{Y}) =1-X_1+(Y_2-X_2)Y_2+(Y_3-Z_3)Y_3,\\
(221^*e)& : (\vect{X}\!-\!\vect{Z})^\top (\vect{Y}\!-\!\vect{Z}) = X_1+X_2Y_2+(Z_3-Y_3)Z_3,\\
\end{split}
\end{equation}
Using auxiliary random variables, we can write the obtusity probability in $(2^*21e)$ sub-configuration as
\begin{equation}
\begin{split}
\eta_{2^*21e} & =\prob{(\vect{Y}-\vect{X})^\top (\vect{Z}-\vect{X})<0}=\prob{-\Omega+\Sigma+\Xi<0} \\
& = \int_0^{1/4}\int_{-1/4}^{\omega} \int_0^{\omega-\sigma} f_\Omega(\omega) f_\Sigma(\sigma) f_\Xi(\xi)  \ddd\xi\dd\sigma\dd \omega \\
& = \int_0^{1/4}\int_{-1/4}^{\omega} f_\Omega(\omega) f_\Sigma(\sigma) F_\Xi(\omega-\sigma) \ddd \sigma\dd \omega.  
\end{split}
\end{equation}
By using the \emph{MSE} integrals, we get
\begin{equation}
    \eta_{2^*21e} =\frac{32629}{54000}-\frac{7G}{30}-\frac{\pi ^2}{360}\approx 0.3630998677.
\end{equation}
Next, in $(22^*1e)$ configuration, we have since $1-X_1\sim \mathsf{Unif}(0,1)$,
\begin{equation}
\eta_{22^*1e} =\prob{(\vect{X}\!-\!\vect{Y})^\top (\vect{Z}\!-\!\vect{Y})<0} =\prob{U\!+\!\Sigma\!+\!\Sigma'\!<\!0} = \eta_{32^*0} = \frac{23}{450}.
\end{equation}
Finally, in $(221^*e)$ configuration,
\begin{equation}
\eta_{221^*e} =\prob{(\vect{X}\!-\!\vect{Z})^\top (\vect{Y}\!-\!\vect{Z})<0} =\prob{U\!+\!\Xi\!+\!\Sigma\!<\!0} = \eta_{31^*1} = \frac{17}{1800}.
\end{equation}
Summing up the three obtusity probabilities,
\begin{equation}
\begin{split}
    \eta_{221r} & = \eta_{2^*21r} + \eta_{22^*1r} + \eta_{221^*r} =
    \frac{35899}{54000}-\frac{7 G}{30}-\frac{\pi ^2}{360} \approx 0.4236554232.
\end{split}
\end{equation}

\phantomsection
\subsection{\texorpdfstring{$\bm{\eta}$\textsubscript{333}}{eta333}}
%\addcontentsline{toc}{subsubsection}{eta333}
Inserting $\eta_{221e}$, $\eta_{221r}$, $\eta_{222r}$, $\eta_{311}$, $\eta_{320}$, $\eta_{321r}$ and $\eta_{322r}$ into Equation \eqref{Eq:CRTcube} with $P = \eta$, for which $p=0$, we finally obtain
\begin{equation}
\begin{split}
  & \eta(C_3) = \eta_{333}  = \frac{1}{28} (4 \eta_{221e}\!+\!\eta_{221r}\!+\!4 \eta_{222r}\!+\!2 \eta_{311}\!+\!2 \eta_{320}\!+\!8 \eta_{321r}\!+\!7 \eta_{322r})\\
  & = \frac{323338}{385875}\!-\!\frac{13G}{35}\!+\!\frac{4859 \pi }{62720}\!-\!\frac{73 \pi }{1680 \sqrt{2}}\!-\!\frac{\pi ^2}{105}\!+\!\frac{3\pi  \ln 2}{224}\!-\!\frac{3\pi  \ln(1\!+\!\sqrt{2})}{224}\\[1ex]
  & \approx 0.54265928142722907450111187258177267165716732602495\ldots,
\end{split}
\end{equation}
which is a natural generalization of Langford's $\eta(C_2)$ \cite{langford1969probability}.

\newpage
\section{Final remarks}

We believe that $\eta(K_3)$ can be found in an exact form for $K_3$ being a box with any side-lengths. CRT equations will be adjusted by the support function weights, but apart of that the calculation should remain similar. Also, since we were successful in dimension three, we wonder whether it is possible to calculate $\eta(C_d)$ for $d \leq 4$ and $\eta(K_3)$ for other three-dimensional polyhedra.

\begin{acknowledgment}{Acknowledgments.}
The work presented is part of my PhD thesis, which I am writing in Prague under the supervision of Jan Rataj. I owe a great deal to his patience and invaluable guidance throughout my whole PhD programme. I deeply appreciate the freedom I was granted, which allowed me to focus fully on my research. I would also like to thank Anna Gusakova and Zakhar Kabluchko for the fruitful discussions we had during my stay in Münster, which greatly contributed to the development of my work.
\end{acknowledgment}

\newpage
% BIBLIOGRAPHY - BIBTEX
%\bibliography{bibliography}

% BIBLIOGRAPHY - BIBLATEX
\printbibliography[heading=bibintoc]

\end{document}